\documentclass{article}
\usepackage[T2A]{fontenc}
\usepackage[cp1251]{inputenc}
\usepackage[russian,english]{babel}
\usepackage[tbtags]{amsmath}
\usepackage{amsfonts,amssymb,mathrsfs,amscd,msb-a,amsmath}
\JournalName{}
\numberwithin{equation}{section}
\theoremstyle{plain}
\newtheorem{conjecture}{Conjecture}
\newtheorem{theorem}{Theorem}
\theoremstyle{plain}

\newtheorem{lemma}{Lemma}
\newtheorem{corollary}{Corollary}
\theoremstyle{definition}
\newtheorem{proof}{Proof}
\newtheorem{definition}{Definition}

\newcommand{\mlegendre}[2]{\left(\frac{#1}{#2}\right)}
\begin{document}

\title{Long nonnegative sums of Legendre symbols}
\author{A.\,B.~Kalmynin}
\address{Steklov Mathematical Institute of Russian Academy of Sciences, Gubkina, 8, Moscow, Russia}
\email{alkalb1995cd@mail.ru}
\date{}
\udk{}
\maketitle
\footnotetext{The work is supported by the Russian Science Foundation under grant \textnumero 19-11-0001.}
\begin{abstract}\begin{bf}{Abstract.}\end{bf}
For $0\leq \alpha<1$ and prime number $p$, let $L(\alpha,p)$ be the sum of the first $[\alpha p]$ values of Legendre symbol modulo $p$. We study positivity of $L(\alpha,p)$ and prove that for $|\alpha-\frac13|<2\cdot 10^{-6}$ and for rational $\alpha\leq \frac12$ with denominators in the set $\{1,2,3,4,5,6,8,12\}$ the inequality $L(\alpha,p)\geq 0$ holds for majority of primes.     
\end{abstract}
\section{Introduction} Let $p$ be an odd prime number and $\chi_p(n)=\mlegendre{n}{p}$ be the Legendre symbol modulo $p$. It is well-known that the sum

\[
\sum_{n\leq p/2}\mlegendre{n}{p}
\]

is always nonnegative. In other words, there are at least as many quadratic residues as nonresidues modulo $p$ below $p/2$. More precisely, Dirichlet proved the following formula for this sum

\begin{theorem}
Let $p$ be an odd prime number. Then the equality

\[
\sum_{n\leq p/2}\mlegendre{n}{p}=\begin{cases}
\left(2-\mlegendre{2}{p}\right)h(-p)\text{ if }p\equiv 3 \pmod 4\\
0 \text{ if }p\equiv 1 \pmod 4
\end{cases}
\]

holds, where $h(-p)$ is the class number of the number field $\mathbb Q(\sqrt{-p}).$
\end{theorem}

Other cases of connection between character sums of this type and class numbers of quadratic fields are given in \cite{suw}. For example, the sum still will be nonnegative if we replace $p/2$ by $p/3$ or $p/4$. This leads us to the general question about nonnegativity of the sum of length $\alpha p$ for any real number $\alpha$. Let us define
\[
L(\alpha,p)=\sum_{n\leq \alpha p}\mlegendre{n}{p}.
\]
Numerical evidence suggests that for any $\alpha\leq \frac{1}{2}$ most primes satisfy the inequality

\begin{equation}
\label{main}    
L(\alpha,p)\geq 0.
\end{equation}
For example, among first $1000$ primes there are $896$ with $L(\frac25,p)\geq 0$, $917$ with $L(\frac38,p)\geq 0$, $884$ with $L(\frac{1}{12},p)\geq 0$, $812$ with $L(\frac{1}{2\pi},p)\geq 0$ and $937$ with $L(\frac{1}{e},p)\geq 0$. For $10000$ these numbers are $8915, 9122, 8799, 8019$ and $9340$, respectively, and we get $89041, 91036, 87868, 79784$ and $93260$ for the first $100000$ prime numbers. As we can see, for all our choices of parameter $\alpha$ the proportion of $p$ with nonnegative $L(\alpha,p)$ seems to be even more than $75\%$. Based on that, let us formulate our main conjecture:
  
\begin{conjecture}
For all $0\leq \alpha\leq \frac{1}{2}$ the lower asymptotic density of primes that satisfy \ref{main} is at least $\frac{1}{2}$. In other words, the inequality \ref{main} holds for majority of prime numbers:
\begin{equation}
\label{goal}
\liminf_{x\to +\infty}\frac{1}{\pi(x)}\#\left\{p\leq x: \sum_{n\leq \alpha p}\mlegendre{n}{p}\geq 0\right\}\geq \frac{1}{2}.
\end{equation}

\end{conjecture}

We were able to prove Conjecture 1 for rational $\alpha$ with small denominators and also for all real $\alpha$ inside a small neighbourhood of the point $\frac{1}{3}$. So, our two main results are as follows:

\begin{theorem}
Conjecture 1 holds for all rational $\alpha\leq \frac{1}{2}$ with denominators in the set $\{1,2,3,4,5,6,8,12\}$.
\end{theorem}
\begin{theorem}
Conjecture 1 holds for all real $\alpha$ satisfying $\frac{1}{3}-2\cdot10^{-6}\leq \alpha\leq \frac{1}{3}+2\cdot10^{-6}$.
\end{theorem}

To prove these two theorems, we are going to reduce the initial problem to the study of certain fixed random variable, using  distribution of primes in arithmetic progressions and various methods of Fourier analysis and probability theory.

\section{Fourier expansion\texorpdfstring{ of $L(\alpha,p)$}{}}

In this section we are going to prove the following simple result:

\begin{theorem}
For any $\alpha$ and all large enough prime numbers $p$ the equality

\begin{equation}
\label{four}
L(\alpha,p)=\tau\left(\mlegendre{\cdot}{p}\right)\sum_{m\in \mathbb Z, m\neq 0}\frac{1-e^{-2\pi i\alpha m}}{2\pi i m}\mlegendre{m}{p}
\end{equation}
holds, where $\tau\left(\mlegendre{\cdot}{p}\right)=\sum\limits_{n=0}^p e^{2\pi in/p}\mlegendre{n}{p}$ is quadratic Gauss sum. 
\end{theorem}
\begin{proof}

Theorem is trivial for $\alpha\in \mathbb Z$, so it is enough to assume that $\alpha$ is not an integer. Also, $L(\alpha,p)$ is a periodic function of $\alpha$ with period 1 and so is the right-hand side of our formula. Thus, we can also assume that $0<\alpha<1.$ 
As periodic characteristic function $\chi_{[0,\alpha]}(\{x\})$ of the interval $[0,\alpha]$ is smooth everywhere except for the discontinuity points, by Dini's critertion for $x \not\equiv 0, \alpha \pmod 1$ we have

\[\chi_{[0,\alpha]}(\{x\})=\sum_{m \in \mathbb Z}\mathcal F\chi_{[0,\alpha]}(m)e^{2\pi i mx},
\]

where

\[
\mathcal F\chi_{[0,\alpha]}(m)=\int_0^\alpha e^{-2\pi i mx}dx=\begin{cases}
\alpha \text{ if } m=0\\
\frac{1-e^{-2\pi i \alpha m}}{2\pi i m} \text{ if } m\neq 0.
\end{cases}
\]

Therefore, due to the fact that $\mlegendre{0}{p}=0$ and $\alpha p$ is not an integer for large enough $p$, the equality

\[
L(\alpha,p)=\sum_{n=0}^p \chi_{[0,\alpha]}(n/p)\mlegendre{n}{p}=\sum_{n=0}^p \sum_{m\in \mathbb Z} \mathcal F\chi_{[0,\alpha]}(m)e^{2\pi nm/p}\mlegendre{n}{p}
\]

holds for $p$ large enough. Changing the order of summation and using the fact that the sum of Legendre symbols over a complete system of residues is equal to 0, we get

\[
L(\alpha,p)=\sum_{m\neq 0}\frac{1-e^{-2\pi i\alpha m}}{2\pi i m}\sum_{n=0}^p e^{2\pi i nm/p}\mlegendre{n}{p}.
\]

From multiplicativity of Legendre symbol, we easily obtain

\[
\sum_{n=0}^p e^{2\pi inm/p}\mlegendre{n}{p}=\tau\left(\mlegendre{\cdot}{p}\right)\mlegendre{m}{p},
\]

from which we get the desired result.
\end{proof}

\section{ Probabilistic reduction}

Here we show that, roughly speaking, in our formula for $L(\alpha,p)$ one can replace all the Legendre symbols by the random multiplicative function. It turns out that it is possible to reduce certain properties of linear combinations of Legendre symbols to properties of random completely multiplicative functions $f$ satisfying $f(n)=\pm 1$ for all $n \in \mathbb N$.

Let us give a few definitions. The main object of our study is the random prime number:

\begin{definition}
For $\varepsilon=\pm 1$ and $x\geq 5$ by $\mathfrak p_x^\varepsilon$ we denote the random variable which is uniformly distributed among primes $\leq x$ that are congruent to $\varepsilon$ modulo 4.
\end{definition}

Next, we need to define the random multiplicative function:

\begin{definition}
Let $X_2,X_3,X_5,X_7,X_{11}\ldots$ be a sequence of independent identically distributed random variables which are distributed according to the Rademacher distribution and indexed by prime numbers. In other words, for any prime $p$ we have
\[
\mathbb P(X_p=1)=\mathbb P(X_p=-1)=\frac12
\]

For any natural number $n$ we define $X_n$ by the formula

\[
X_n=\prod_{p} X_p^{\nu_p(n)},
\]

where $\nu_p(n)$ is the largest $k$ with $p^k \mid n$. Note that the product contains finite number of terms and also that $X_{ab}=X_aX_b$ for all $a$ and $b$.
\end{definition}

Using constructed random variables, we define certain random series.

\begin{definition}
Let $\{a_n\}$ be a sequence of complex numbers, $x\geq 5$ and $\varepsilon=\pm 1$. We define

\[
L(a,x,\varepsilon)=\sum_{n=1}^{+\infty}\frac{a_n}{n}\mlegendre{n}{\mathfrak p_x^\varepsilon}
\]

and

\[
L(a)=\sum_{n=1}^{+\infty}\frac{a_nX_n}{n}.
\]
\end{definition}

In this section we will show that $L(a)$ is often a rather good model for $L(a,x,\varepsilon)$. But first of all, we need to show that $L(a)$ is well-defined.

\begin{lemma}
If the sequence $\{a_n\}$ is bounded, then the series defining $L(a)$ converges almost surely.
\end{lemma}

\begin{proof}
Assume that $|a_n|\leq C$ for all $n$. Consider the fourth moment of the weighted sum of $X_n$:

\[
\mathbb E(|a_1X_1+\ldots+a_nX_n|^4)=\sum_{p,q,r,s\leq n}a_pa_q\overline{a}_r\overline{a}_s\mathbb EX_{pqrs}.
\]

Using the boundedness of $a_n$ and noticing that $\mathbb EX_n=0$ unless $n$ is a square, in which case $\mathbb EX_n=1$, we get

\[
\mathbb E(|a_1X_1+\ldots+a_nX_n|^4)\leq C^4 \sum_{\substack{pqrs=m^2 \\ p,q,r,s\leq n}} 1.
\]

Now, if $pqrs=m^2$ and $p,q,r,s\leq n$ then $m\leq n^2$, so

\[\mathbb E(|a_1X_1+\ldots+a_nX_n|^4)\leq C^4\sum_{m\leq n^2}\tau_4(m^2).
\]

As $\tau_4(m^2)\ll m^\varepsilon$ for any $\varepsilon>0$ we obtain

\[
\mathbb E(|a_1X_1+\ldots+a_nX_n|^4)\ll n^{2+\varepsilon}.
\]

Choosing $\varepsilon=1/6$ we get by Markov's inequality

\[
\mathbb P(|a_1X_1+\ldots+a_nX_n|\geq n^{5/6})\ll \frac{n^{2+1/6}}{n^{20/6}}=n^{-7/6}.
\]

Hence, by Borel-Cantelli lemma we have $a_1X_1+\ldots+a_nX_n=O(n^{5/6})$ almost surely. Using partial summation, we obtain the convergence of

\[
\sum_{n\geq 1}\frac{a_nX_n}{n}.
\]
\end{proof}

Now we are going to use some results on primes in arthmetic progressions to prove the following theorem

\begin{theorem}
Let $\{a_n\}$ be a bounded sequence of real numbers such that the inequality

\[
\max_N \sum_{n\leq N}a_n\mlegendre{n}{p}\ll \sqrt{p}\ln p
\]

holds for all but at most $o(\pi(x))$ primes $p\leq x$ as $x\to +\infty$ (i.e. abovementioned inequality is true for almost all primes). Then for $\varepsilon=\pm 1$ random variables $L(a,x,\varepsilon)$ converge to $L(a)$ in distribution as $x \to +\infty$.
\end{theorem}

\begin{proof}

The proof will be divided into several parts in which we will treat different chunks of our series differently. We will formulate and use several lemmas concerning the distribution of prime numbers and the method of moments inside the proof. First of all, we split the sum in the definition of $L(a,x,\varepsilon)$ into three sums as follows:

\[
L(a,x,\varepsilon)=\sum_{n\leq\ln^3 x}\frac{a_n}{n}\mlegendre{n}{\mathfrak p_x^\varepsilon}+\sum_{\ln^3 x<n\leq \sqrt{x}\ln^2 x}\frac{a_n}{n}\mlegendre{n}{\mathfrak p_x^\varepsilon}+\sum_{n>\sqrt{x}\ln^2 x}\frac{a_n}{n}\mlegendre{n}{\mathfrak p_x^\varepsilon}=
\]

\[
A(a,x,\varepsilon)+B(a,x,\varepsilon)+C(a,x,\varepsilon).
\]

We will prove that as $x\to +\infty$ the variables $B(a,x,\varepsilon)$ and $C(a,x,\varepsilon)$ both converge to 0 in probability and then show that $A(a,x,\varepsilon)$ converges to $L(a)$ in distribution via the method of moments.

Let us prove that the infinite sum $C(a,x,\varepsilon)$ is usually small. This easily follows from the conditions of Theorem 5 because with probability $1-o(1)$ we have $\sum\limits_{n\leq N}a_n\mlegendre{n}{p}\ll \sqrt{p}\ln p$, so for almost all realizations $p$ of $\mathfrak p_x^\varepsilon$ we get

\[
\sum_{n>\sqrt{x}\ln^2 x}\frac{a_n}{n}\mlegendre{n}{p}=\int\limits_{\sqrt{x}\ln^2 x}\frac{d\left(\sum_{n\leq t}a_n\mlegendre{n}{p}\right)}{t}=-\frac{\sum_{n\leq \sqrt{x}\ln^2 x}a_n\mlegendre{n}{p}}{\sqrt{x}\ln^2 x}+
\]
\[
\int\limits_{\sqrt{x}\ln^2 x}\frac{\sum_{n\leq t}a_n\mlegendre{n}{p}}{t^2}dt=O\left(\frac{\sqrt{p}\ln p}{\sqrt{x}\ln^2 x}\right)=O\left(\frac{1}{\ln x}\right),
\]

as $p\leq x$. Therefore, $C(a,x,\varepsilon)$ converges to 0 in probability as $x\to +\infty$. To handle $B(a,x,\varepsilon)$, let us estimate the expectation of $B(a,x,\varepsilon)^2$. Notice the following property of $\mlegendre{n}{p}$:

\begin{lemma}
For any nonzero integer $a$ there is a Dirichlet character $\chi_a$ of modulus at most $4|a|$ such that for any odd integer $n$ we have $\mlegendre{a}{n}=\chi_a(n)$ and $\chi_a$ is nonprincipal if and only if $a$ is not a square.
\end{lemma}
\begin{proof}
For any $a$ there are $b$ and $c$ such that $a=bc^2$, $b$ is a fundamental discriminant and $2c$ is an integer. Now, for any odd integer $n$ we have

\[
\mlegendre{a}{n}=\mlegendre{b}{n}\mlegendre{c^2}{n}=\mlegendre{b}{n}\chi_{0,2c}(n).
\]

As $b$ is a fundamental discriminant, $\mlegendre{b}{n}$ is a Dirichlet character to the modulus $|b|$ (see \cite{IK}, p. 53). Therefore $\mlegendre{a}{n}=\mlegendre{b}{n}\chi_{0,2c}(n)=\chi_a(n)$, where $\chi_a$ has modulus at most $2|b|c\leq 4|b|c^2=4a$, which completes the proof.
\end{proof}

Now, the expectation of $B(a,x,\varepsilon)^2$ equals

\[
\frac{1}{\pi(x,4,\varepsilon)}\sum_{\substack{p\leq x \\ p\equiv \varepsilon\pmod 4}} \left|\sum_{\ln^3 x<n\leq \sqrt{x}\ln^2 x}\frac{a_n}{n}\mlegendre{n}{p}\right|^2
\]

Due to nonnegativity of summands, we can sum over all odd numbers instead of primes and get

\[
\mathbb EB(a,x,\varepsilon)^2\leq \frac{1}{\pi(x,4,\varepsilon)}\sum_{\substack{d\leq x \\ 2\nmid d}}\left|\sum_{\ln^3 x<n\leq \sqrt{x}\ln^2 x} \frac{a_n}{n}\mlegendre{n}{d}\right|^2=
\]

\[
=\frac{1}{\pi(x,4,\varepsilon)}\sum_{\substack{d\leq x \\ 2\nmid d}} \sum_{\ln^3 x<n,k<\sqrt{x}\ln^2 x}\frac{a_na_k}{nk}\mlegendre{nk}{d}=\frac{1}{\pi(x,4,\varepsilon)}\sum_{\substack{d \leq x \\ 2\nmid d}}\sum_{\ln^6 x<m<x\ln^4 x}\frac{b_m}{m}\mlegendre{m}{d},
\]

where $b_m=\sum\limits_{\ln^3 x<n,k<\sqrt{x}\ln^2 x}a_na_k$. As $|a_n|=O(1)$ for all $n$, we have $b_m=O(\tau(m))$. Changing the order of summation and using Lemma 2 we deduce the inequality

\[
\mathbb EB(a,x,\varepsilon)^2\leq \frac{1}{\pi(x,4,\varepsilon)}\sum_{\ln^6 x<m<x\ln^4 x}\frac{|b_m|}{m}\left|\sum_{d\leq x, 2\nmid d} \chi_m(d)\right|
\]

Now, if $m$ is a square, then inner sum is trivially estimated by $x$, while if it isn't a square, the P\'olya-Vinogradov inequality gives us the bound $O(\sqrt{m}\ln m)$. Also, $\pi(x,4,\varepsilon)\sim \frac{x}{2\ln x}$, therefore

\[
\mathbb EB(a,x,\varepsilon)^2\ll \frac{\ln x}{x}\left(x\sum_{\ln^3 x<l<\sqrt{x}\ln^2 x} \frac{\tau(l^2)}{l^2}+\sum_{\ln^6 x<m<x\ln^4 x}\frac{\tau(m)\ln m}{\sqrt{m}}\right).
\]

As for any $\delta>0$ the inequality $\tau(m)\ll m^\delta$ holds, we obtain

\[
\sum_{\ln^3 x<l<\sqrt{x}\ln^2 x}\frac{\tau(l^2)}{l^2}\ll \frac{1}{\ln^2 x}
\]

and

\[
\sum_{\ln^6 x<m<x\ln^4 x}\frac{\tau(m)\ln m}{\sqrt{m}}\ll x^{2/3}.
\]

These bounds result in the estimate

\[
\mathbb EB(a,x,\varepsilon)^2 \ll \frac{1}{\ln x},
\]

which implies that $B(a,x,\varepsilon)$ converges to $0$ in probability as $x\to +\infty$.

So, we are left with the shortest part of our sum. To prove that $A(a,x,\varepsilon)$ converges to $L(a)$ in distribution, we are going to use the moment method in the following form:

\begin{lemma}[(Carleman's criterion)]
Let $\xi_n$ be a sequence of real random variables such that for some random variable  $\xi$ and all natural numbers $k$ the identity
\[
\lim_{n\to +\infty}\mathbb E\xi_n^k=\mathbb E\xi^k
\]

holds.
Assume that $\xi$ satisfies Carleman's criterion

\[
\sum_{n\geq 1}(\mathbb E\xi^{2n})^{-1/2n}=+\infty.
\]

Then $\xi_n$ converges to $\xi$ in distribution for $n\to +\infty$.
\end{lemma}
\begin{proof}
See \cite{Akh}, p. 85.
\end{proof}
In the moment computation we will use the following classical result on primes in arithmetic progressions

\begin{lemma}[(Siegel-Walfisz theorem)]
Let $A>0$ be a fixed real number. Then there is a positive constant $c_A$ such that for any real $x>0$ and nonprincipal Dirichlet character $\chi$ modulo $q\leq (\ln x)^A$ the estimate

\[
\pi(x;\chi)=\sum_{p\leq x}\chi(p)=O\left(xe^{-c_A\sqrt{\ln x}}\right)
\]

holds.
\end{lemma}
\begin{proof}
See \cite{K}, p. 138.
\end{proof}

As an immediate consequence, we obtain the following result on expectations of Legendre symbols:

\begin{corollary}
For any positive integer $k$ there is a positive constant $c_k$ such that for all $0<n\leq (\ln x)^k$ and $\varepsilon=\pm 1$ we have

\[
\mathbb E\mlegendre{n}{\mathfrak{p}_x^\varepsilon}=\square(n)+O(\exp(-c_k\sqrt{\ln x})),
\]

where $\square(n)=1$ if $n$ is a square and $0$ otherwise.
\end{corollary}
\begin{proof}
Indeed, this expectation can be rewritten as

\[
\mathbb E\mlegendre{n}{\mathfrak{p}_x^\varepsilon}=\frac{1}{\pi(x,4,\varepsilon)}\sum_{\substack{p\leq x \\ p\equiv \varepsilon \pmod 4}}\mlegendre{n}{p}=\frac{1}{2\pi(x,4,\varepsilon)}\sum_{p\leq x}\left(\mlegendre{n}{p}+\varepsilon\mlegendre{-4n}{p}\right).
\]

Applying Lemma 2 and Lemma 4, we get the desired result.
\end{proof}

Using Corollary 1, one can easlily deduce the formula for $k-$th moment of $A(a,x,\varepsilon)$. Note first that

\[
A(a,x,\varepsilon)^k=\sum_{n_i\leq \ln^3 x}\frac{a_{n_1}a_{n_2}\ldots a_{n_k}}{n_1\ldots n_k}\mlegendre{n_1\ldots n_k}{\mathfrak{p}_x^\varepsilon}.
\]

Rearranging the summands according to the product of variables, we get

\[
A(a,x,\varepsilon)^k=\sum_{n\leq \ln^{3k}x}\frac{\tau_k(n;a,x)}{n}\mlegendre{n}{\mathfrak{p}_x^\varepsilon},
\]

where $\tau_k(n;a,x)=\sum\limits_{n_1\ldots n_k=n, n_i\leq \ln^3 x} a_{n_1}\ldots a_{n_k}$. Also, one can easily show that $|\tau_k(n;a,x)|\ll \tau_k(n)$ and $\tau_k(n;a,x)\to \tau_k(n;a)=\tau_k(n;a,\infty)$. Using the linearity of expectation and Corollary 1, we deduce that

\[
\mathbb EA(a,x,\varepsilon)^k=\sum_{n\leq \ln^k x}\frac{\tau_k(n^2;a,x)}{n^2}+O(\exp(-0.5c_k\sqrt{\ln x})).
\]

Thus, by dominated convergence theorem we obtain

\[
\lim_{x\to +\infty}\mathbb EA(a,x,\varepsilon)^k=\sum_{n}\frac{\tau_k(n^2;a)}{n^2}.
\]

Now, using dyadic subdivision, one can prove that the $k-$th moment of $L(a)$ exists and equals

\[
\mathbb EL(a)^k=\sum_{n_1,\ldots,n_k}\frac{a_{n_1}\ldots a_{n_k}\mathbb EX_{n_1\ldots n_k}}{n_1\ldots n_k}=\sum_n \frac{\tau_k(n^2;a)}{n^2},
\]

because the expectation in this sum is only nonzero when $n_1\ldots n_k$ is a square. We now need to prove that $L(a)$ satisfies Carleman's condition. To do so, notice that $|\tau_k(n^2;a)|\leq C^k \tau_k(n^2)$. Therefore,

\[
EL(a)^{2k}\leq C^{2k}\sum_n \frac{\tau_{2k}(n^2)}{n^2}.
\]

As $\tau_{2k}(n^2)$ is multiplicative, we get

\[
EL(a)^{2k}\leq C^{2k}\prod_p\left(\frac{1}{2}\left(\frac{1}{(1-p^{-1})^{2k}}+\frac{1}{(1+p^{-1})^{2k}}\right)\right).
\]

Let us split this product in two parts. For primes $p\leq k^2$ we use the inequality

\[
\frac{1}{2}\left(\frac{1}{(1-p^{-1})^{2k}}+\frac{1}{(1+p^{-1})^{2k}}\right)\leq \frac{1}{(1-p^{-1})^{2k}},
\]

from which we deduce by Mertens' third theorem that

\[
\prod_{p\leq k^2} \left(\frac{1}{2}\left(\frac{1}{(1-p^{-1})^{2k}}+\frac{1}{(1+p^{-1})^{2k}}\right)\right)\leq \prod_{p\leq k^2}\frac{1}{(1-p^{-1})^{2k}}\leq B^{2k}(\ln k)^{2k}
\]

for some absolute constant $B$. Further, for $p>k^2$ the relation
\[
\frac{1}{2}\left(\frac{1}{(1-p^{-1})^{2k}}+\frac{1}{(1+p^{-1})^{2k}}\right)=\frac{1+p^{-2}{2k\choose 2}+p^{-4}{2k\choose 4}+\ldots+p^{-2k}{2k\choose 2k}}{(1-p^{-2})^{2k}}
\]
holds.

Using the inequality ${2k\choose m}\leq (2k)^m$ for all $m$, we obtain for $k>2$ the estimate

\[
\frac{1}{2}\left(\frac{1}{(1-p^{-1})^{2k}}+\frac{1}{(1+p^{-1})^{2k}}\right)\leq \frac{1+8p^{-2}k^2}{(1-p^{-2})^{2k}}
\]

Thus for all $k>2$ we get
\[
\prod_{p>k^2} \left(\frac{1}{2}\left(\frac{1}{(1-p^{-1})^{2k}}+\frac{1}{(1+p^{-1})^{2k}}\right)\right)\leq \prod_{p>k^2}\frac{1+8p^{-2}k^2}{(1-p^{-2})^{2k}}
\]
\[\leq \left(\frac{\pi^2}{6}\right)^{2k}\exp\left(\sum_{p>k^2}\frac{8k^2}{p^2}\right)\ll\left(\frac{\pi^2}{6}\right)^{2k}.
\]

Combining these bounds we obtain for some fixed $B$ and $C$

\[
\mathbb EL(a)^{2k}\ll \left(\frac{\pi^2}{6}BC\ln k\right)^{2k}.
\]

This upper bound implies that

\[
\sum_{k\geq 1}(\mathbb EL(a)^{2k})^{-1/2k}\gg \sum_{k\geq 3}\frac{1}{\ln k}=+\infty.
\]

Therefore $L(a)$ satisfies Carleman's condition and this completes the proof.
\end{proof} 

Theorem 5 can be used in very different settings, but we are going to use the following corollary:

\begin{corollary}
If $a_n=\sum_m \lambda_m e^{2\pi i n\alpha_m}$ for some finite sequence of complex $\lambda_m$ and real $\alpha_m$ and $a_n$ is real for all $n$, then $L(a,x,\varepsilon)$ converges to $L(a)$ in distribution.
\end{corollary}
\begin{proof}
Obviously, $a_n$ is bounded. Therefore, it is enough to check that for any real $\alpha$ the estimate

\[
\sum_{n\leq N}e^{2\pi i \alpha n}\mlegendre{n}{p}\ll \sqrt{p}\ln p
\]

holds for almost all primes $p$. To show that this is indeed the case, note that if $\alpha$ is rational and nonzero then $||\alpha p||\gg 1$ for $p$ large enough and if $\alpha$ is irrational then $\alpha p_n$ is uniformly distributed modulo 1 (see \cite[p.489, Theorem 21.3]{IK}), so for almost all $p$ the inequality $||\alpha p||\geq \frac{1}{\ln p}$ holds. Using the Fourier expansion for the Legendre symbol, we get for almost all $p$

\[
\sum_{n\leq N}e^{2\pi i \alpha n}\mlegendre{n}{p}=\frac{1}{\tau\left(\mlegendre{\cdot}{p}\right)}\sum_{a=1}^{p-1}\sum_{n\leq N}\exp(2\pi in(\alpha-a/p))\ll \frac{1}{\sqrt{p}}\sum_{a=1}^{p-1}||\alpha-a/p||^{-1}\ll
\]

\[
\ll \frac{1}{\sqrt{p}}\sum_{a=-p/2}^{p/2}\frac{p}{|a+1/\ln p|}\ll\sqrt{p}\ln p, 
\]

which completes the proof for $\alpha\neq 0$. If $\alpha=0$, then the desired bound is just the P\'olya-Vinogradov inequality.
\end{proof}

Corollary 2 has a very useful implication on our main problem. Namely, from Section 2 we deduce that we need to study the distribution of $L(a^{\pm},x,\pm1)$ with $a^{\pm}_n$ being an exponential polynomial of variable $n\alpha.$ We are also going to use the following well-known formula for the Gauss sum:

\begin{lemma}
Let $p$ be an odd prime number. Then we have

\[
\tau\left(\mlegendre{\cdot}{p}\right)=\begin{cases}
\sqrt{p}\text{ if }p\equiv 1\pmod 4\\
i\sqrt{p}\text{ if }p\equiv 3\pmod 4.
\end{cases}
\]
\end{lemma}
\begin{proof}
See \cite{IK}, p. 49.
\end{proof}

From this we get the following:

\begin{corollary}
Let $\alpha$ be a real number, $a^{+}_n(\alpha)=\sin 2\pi n\alpha$ and $a^{-}_n(\alpha)=1-\cos2\pi n\alpha$. Let $c(\alpha)$ be a relative lower density of primes $p$, satisfying $L(\alpha,p)\geq 0$, $c^{+}(\alpha)$ and $c^{-}(\alpha)$ be the probabilities of positivity of random variables $L(a^{+}(\alpha))$ and $L(a^{-}(\alpha))$. Then we have

\[
c(\alpha)\geq \frac{c^{+}(\alpha)+c^{-}(\alpha)}{2}.
\]

In particular, if $c^{+}(\alpha)+c^{-}(\alpha)>1$, then Conjecture 1 is true for $\alpha.$
\end{corollary}

\begin{proof}
First, it is enough to only consider large enough odd prime numbers, because we can ignore any finite number of primes. Now, if $p\equiv 1 \pmod 4$ then we have  by Theorem 4

\[
L(\alpha,p)=\tau\left(\mlegendre{\cdot}{p}\right)\sum_{m\in \mathbb Z, m\neq 0}\frac{1-e^{-2\pi i \alpha m}}{2\pi i m}\mlegendre{m}{p}=
\]
\[=\sqrt{p}\sum_{m> 0}\left(\frac{1-e^{-2\pi i \alpha m}}{2\pi i m}-\frac{1-e^{2\pi i \alpha m}}{2\pi i m}\right)\mlegendre{m}{p}=\frac{\sqrt{p}}{\pi}\sum_{m>0}\frac{\sin 2\pi m\alpha}{m}\mlegendre{m}{p},
\]

because $\mlegendre{-m}{p}=\mlegendre{-1}{p}\mlegendre{m}{p}=\mlegendre{m}{p}.$ Therefore, the proportion of primes $p\leq x$ with $p\equiv 1 \pmod p$ with $L(\alpha,p)\geq 0$ among all primes $p\leq x$ is

\[
\frac{\pi(x;4,1)}{\pi(x)}\mathbb P(L(a^{+},x,1)\geq 0)\sim \frac{\mathbb P(L(a^{+},x,1)\geq 0)}{2}
\]

as $x\to +\infty$. Now, as $L(a^{+},x,1)$ converges to $L(a^{+})$ in distribution and the subset $\mathbb R_{>0} \subset \mathbb R$ is open, we have

\[
\liminf_{x\to +\infty}\mathbb P(L(a^{+},x,1)\geq 0)\geq \liminf_{x\to +\infty}\mathbb P(L(a^{+},x,1)>0)\geq \mathbb P(L(a^{+})>0)=c^{+}(\alpha).
\]

If $p\equiv 3 \pmod 4$, then we have similarly

\[
L(\alpha,p)=\tau\left(\mlegendre{\cdot}{p}\right)\sum_{m\in \mathbb Z, m\neq 0}\frac{1-e^{-2\pi i \alpha m}}{2\pi i m}\mlegendre{m}{p}=
\]
\[=i\sqrt{p}\sum_{m> 0}\left(\frac{1-e^{-2\pi i \alpha m}}{2\pi i m}+\frac{1-e^{2\pi i \alpha m}}{2\pi i m}\right)\mlegendre{m}{p}=\frac{\sqrt{p}}{\pi}\sum_{m>0}\frac{1-\cos 2\pi m\alpha}{m}\mlegendre{m}{p}.
\]

Hence, the proportion of primes $p\leq x$ with $p\equiv 3 \pmod p$ with $L(\alpha,p)\geq 0$ among all primes $p\leq x$ is

\[
\frac{\pi(x;4,3)}{\pi(x)}\mathbb P(L(a^{-},x,-1)\geq 0)\sim \frac{\mathbb  P(L(a^{-},x,-1)\geq 0)}{2}.
\]

And we also have

\[
\liminf_{x\to +\infty}\mathbb P(L(a^{-},x,1)\geq 0)\geq \liminf_{x\to +\infty}\mathbb P(L(a^{-},x,1)>0)\geq \mathbb P(L(a^{-})>0)=c^{-}(\alpha).
\]

Therefore, we obtain

\[
\liminf_{x\to +\infty}\frac{\#\{p\leq x: L(\alpha,p)\geq 0\}}{\pi(x)}\geq \frac{c^{+}(\alpha)+c^{-}(\alpha)}{2},
\]

as needed.

\end{proof} 

\section{Rational \texorpdfstring{$\alpha$}{alpha} with small denominators}

In this section we are going to prove Theorem 2. In other words, we are going to prove Conjecture 1 for the following values of $\alpha$:

\[
\alpha=0,\frac{1}{2},\frac{1}{3},\frac{1}{4},\frac{1}{6},\frac{1}{8},\frac{3}{8},\frac{1}{12},\frac{5}{12},\frac{1}{5},\frac{2}{5}.
\]

All the denominators of these numbers, except for $5$, are precisely the numbers $n$ such that the group $\left(\mathbb Z/n\mathbb Z\right)^*$ has exponent $1$ or $2$, i.e. such that $a^2\equiv 1\pmod n$ for all $a$ coprime to $n$. The reason for such a choice of denominators is that for these $n$ all the Dirichlet characters modulo $n$ are real-valued and so we don't need to consider any complex Euler products. On the other hand, in the case of $\alpha=\frac{1}{5}$ or $\frac{2}{5}$ we need to compute arguments of certain complex-valued Euler products, which is going to make things more complicated. Probably, the list of $n$ such that Conjecture 1 follows from conditions $n\alpha\in \mathbb Z$ and $\alpha<\frac{1}{2}$ can be expanded in a way similar to what is presented in this section, but we don't know if it works for all rational $\alpha$.

For a natural number $m$ the function $\chi_{0,m}(\cdot)$ will be the principal character modulo $m$. If $\chi$ is a Dirichlet character and $\beta$ is not an integer then we are going to define $\chi(\beta)$ to be equal $0$. So, for example, $\chi_{0,2}\left(\frac{n}{2}\right)$ is a 4-periodic function with values $0,1,0,0,0,1,\ldots$
 
\begin{proof}[of Theorem 2]
Let us start with a few simple cases.

If $\alpha=0$ then our conjecture is trivial, as we always have $L(0,p)=0$, so $c(0)=1$.

If $\alpha=\frac{1}{2}$, then

\[
a^{+}_n(\alpha)=\sin \pi n=0
\]

for all $n$ and

\[
a^{-}_n(\alpha)=1-\cos\pi n=\begin{cases}
2 \text{ if }n\text{ is odd}\\
0 \text{ otherwise}
\end{cases}
\]

From these formulas we get

\[
L(a^{+})=0
\]

for all possible realizations of $(X_n)$. Therefore, for all primes $p\equiv 1\pmod 4$ we have $L(1/2,p)=0$. Indeed, in the previous section we learned that for $p \equiv 1 \pmod 4$ the quantity $\frac{\pi}{\sqrt{p}}L(\alpha,p)$ is a particular realization of random variable $L(a^{+})$, while for $p\equiv 3 \pmod 4$ it is a realization of random variable $L(a^{-})$. From our formulas we also get

\[
L(a^{-})=2\sum_{n>0}\frac{\chi_{0,2}(n)X_n}{n}=2\prod_{p}\left(1-\frac{\chi_{0,2}(p)X_p}{p}\right)^{-1}\geq 0,
\]

due to multiplicativity of $\chi_{0,2}(n)X_n$. Therefore, for all primes $p$ we have $L(1/2,p)\geq 0$, so $c(1/2)=1$.

Now, for $\alpha=\frac{1}{3}$ one can easily check that

\begin{equation}
\label{aplus3}    
a_n^{+}(1/3)=\sin\frac{2\pi n}{3}=\frac{\sqrt{3}}{2}\mlegendre{n}{3}
\end{equation}

and

\begin{equation}
\label{amin3}    
a_n^{-}(1/3)=1-\cos\frac{2\pi n}{3}=\frac{3}{2}\chi_{0,3}(n).
\end{equation}

and we again get nonnegativity of all realizations from the Euler products for corresponding series. Hence in this case we again have $c(\alpha)=c(1/3)=1$.

The same is true for $\alpha=\frac{1}{4}$, because in that case we get

\[
a_n^{+}(1/4)=\sin\frac{\pi n}{2}=\mlegendre{-4}{n}=\chi_4(n)
\]

and

\[
a_n^{-}(1/4)=1-\cos\frac{\pi n}{2}=\chi_{0,2}(n)+2\chi_{0,2}\left(\frac{n}{2}\right),
\]

therefore $c(1/4)=1$, because

\[
L(a^{+})=\prod_p \left(1-\frac{\chi_4(p)X_p}{p}\right)^{-1}\geq 0
\]

and

\[
L(a^{-})=\prod_p \left(1-\frac{\chi_{2,0}(p)X_p}{p}\right)^{-1}+X_2\prod_p \left(1-\frac{\chi_{0,2}(p)X_p}{p}\right)^{-1}=
\]
\[
(1+X_2)\prod_{p>2} \left(1-\frac{X_p}{p}\right)^{-1}\geq 0.
\]
The factor $X_2$ appears here because for any $\chi$ and any $d$ we have

\[
\sum_{n}\frac{\chi(n/d)X_n}{n}=\sum_{n}\frac{\chi(n)X_{nd}}{nd}=\frac{X_d}{d}\sum_n \frac{\chi(n)X_n}{n}.
\]

For $\alpha=\frac{1}{6}$ we once again obtain $c(1/6)=1$, but for a more complicated reasons.

Here we have

\[
a_n^{+}(\alpha)=\sin\frac{\pi n}{3}.
\]

The values of $a_n^{+}$ are $\frac{\sqrt{3}}{2},-\frac{\sqrt{3}}{2},0,-\frac{\sqrt{3}}{2},\frac{\sqrt{3}}{2},0,\ldots$ (the sequence is 6-periodic)
so we get

\[
a_n^{+}(\alpha)=\frac{\sqrt{3}}{2}\left(\chi_6(n)+\mlegendre{n/2}{3}\right).
\]

Here $\chi_6(n)=\pm 1$ if $n\equiv \pm 1\pmod 6$ and 0 otherwise. Note also that $\chi_6(n)=\chi_{0,2}(n)\mlegendre{n}{3}.$ Therefore,

\[L(a^{+})=\frac{\sqrt{3}}{2}\left(\sum_n\frac{\chi_{0,2}(n)\mlegendre{n}{3}X_n}{n}+\sum_n\frac{\mlegendre{n/2}{3}X_n}{n}\right)=\]

\[
\frac{\sqrt{3}}{2}(1+X_2)\prod_{p}\left(1-\frac{\mlegendre{p}{3}X_p}{p}\right)^{-1}\geq 0,
\]

so we again have $L(\alpha,p)\geq 0$ for all $p\equiv 1 \pmod 4$.

Values of $a_n^{-}$ are $\frac{1}{2},\frac{3}{2},2,\frac{3}{2},\frac12,0\ldots$ and this function can be expanded as follows:

\[
a_n^{-}=1-\cos\frac{\pi n}{3}=2\chi_{0,2}\left(\frac{n}{3}\right)+\frac{1}{2}\chi_{0,3}(n)+\chi_{0,3}\left(\frac{n}{2}\right).
\]

Nonnegativity is not obvious from this formula, so let us work with the Euler products

\[
L(a^{-})=\sum_{n}\frac{a_n^{-}}{n}=\frac{2X_3}{3}\sum_n \frac{\chi_{0,2}(n)X_n}{n}+\frac{1+X_2}{2} \sum_{n}\frac{\chi_{0,3}(n)X_n}{n}=
\]
\[
=\left(\frac{2X_3}{3}\left(1-\frac{X_2}{2}\right)+\frac{1+X_2}{2}\left(1-\frac{X_3}{3}\right)\right)\sum_n \frac{X_n}{n}=\frac{1+X_2+X_3-X_2X_3}{2}\sum_n \frac{X_n}{n}\geq 0.
\]
This calculation shows that $c(1/6)=1.$

Now we are going to consider a different set of values $\alpha$ for which we don't have $c(\alpha)=1$.

If $\alpha=\frac18$ then the first $8$ values of $a_n^{+}(1/8)=\sin\frac{\pi n}{4}$ are equal to \[\frac{\sqrt{2}}{2},1,\frac{\sqrt{2}}{2},0,-\frac{\sqrt{2}}{2},-1,-\frac{\sqrt{2}}{2},0,\] which is equal to $\frac{\sqrt{2}}{2}\mlegendre{-2}{n}+\chi_4\left(\frac{n}{2}\right)$ thus we obtain

\[
L(a^{+})=\frac{\sqrt{2}}{2}\sum_n\frac{\mlegendre{-2}{n}X_n}{n}+\frac{X_2}{2}\sum_n \frac{\chi_4(n)X_n}{n}
\]

and hence $c^{+}(1/8)>\frac12$, because for $X_2=+1$ we have $L(a^{+})\geq 0$ and for $X_2=-1$ our random variable is positive with positive probability, because

\[
\mathbb E(L(a^{+})\mid X_2=-1)=\frac{\sqrt{2}-1}{2}\sum_{2\nmid n}\frac{1}{n^2}=\frac{(\sqrt{2}-1)\pi^2}{18}>0.
\]

On the other hand, the first $8$ values of $a_n^{-}(1/8)$ are 
\[
1-\frac{\sqrt{2}}{2},1,1+\frac{\sqrt{2}}{2},2,1+\frac{\sqrt{2}}{2},1,1-\frac{\sqrt{2}}{2},0
\]

and one can check that we have

\[
a_n^{-}(1/8)=\chi_{0,2}(n)+\chi_{0,2}(n/2)+2\chi_{0,2}(n/4)-\frac{\sqrt{2}}{2}\mlegendre{2}{n}.
\]

Therefore

\[
L(a^{-})=\left(\frac32+\frac{X_2}{2}\right)\sum_n \frac{\chi_{0,2}(n)X_n}{2}-\frac{\sqrt{2}}{2}\sum_n \frac{\mlegendre{2}{n}X_n}{n}\geq
\]
\[
\geq \sum_n \frac{\chi_{0,2}(n)X_n}{n}-\frac{\sqrt{2}}{2}\sum_n \frac{\mlegendre{2}{n}X_n}{n}=F-\frac{\sqrt{2}}{2}G.
\]

Now, let us notice, that the distribution of $(\lambda(n)X_n)$, where $\lambda(n)$ is the Liouville's function, coincides with the distribution of $(X_n)$, as $-X_p$ are also independent and have Rademacher's distribution. On the other hand, if we replace $X_n$ by $\lambda(n)X_n$, $F$ and $G$ will transform into

\[
\sum_n \frac{\lambda(n)\chi_{0,2}(n)X_n}{n}=\prod_{p>2}\left(1+\frac{X_p}{p}\right)^{-1}=\prod_{p>2}\left(1-\frac{1}{p^2}\right)^{-1}\left(1-\frac{X_p}{p}\right)=\frac{\pi^2}{9F}
\]

and (by the same argument) $\frac{\pi^2}{9G}$, respectively. Also, $\ln F-\ln G$ has continuous distribution (because characteristic function goes to $0$ at infinity), so $\mathbb P(L(a^{-})>0)\geq \frac12$, because $F-\frac{\sqrt{2}}{2}G\geq 0$ holds either for $X_n$ or for $\lambda(n)X_n$, so we get $c(1/8)>\frac12$.

As for the second case of denominator $8$, i.e. $\alpha=\frac38$, using $a_{3n}^{\pm}(1/8)=a_{n}^{\pm}(3/8)$ we get

\[
L(a^{+}(3/8))=\frac{\sqrt{2}}{2}\sum_n\frac{\mlegendre{-2}{n}X_n}{n}-\frac{X_2}{2}\sum_n \frac{\chi_4(n)X_n}{n},
\]

so $c^{+}(3/8)=c^{+}(1/8)>\frac12$. Also, in the formula for $L(a^{-})$ we will get $+G$ instead of $-G$, so that

\[
L(a^{-})\geq F+\frac{\sqrt{2}}{2}G\geq 0.
\]

Hence we get $c^{-}(3/8)=1$ and $c(3/8)\geq \frac34$.

The last two cases without any complex Dirichlet characters are $\alpha=\frac{1}{12}$ and $\alpha=\frac{5}{12}$.

Let us start with $\alpha=\frac{1}{12}$. For $a_n^{+}=\sin\frac{\pi n}{6}$ we get the first twelve values

\[
\frac12,\frac{\sqrt{3}}{2},1,\frac{\sqrt{3}}{2},\frac12,0,-\frac12,-\frac{\sqrt{3}}{2},-1,-\frac{\sqrt{3}}{2},-\frac12,0,
\]

find the formula

\[
a_n^{+}(1/12)=\frac12 \chi_4(n)\chi_{0,3}(n)+\frac{\sqrt{3}}{2}\mlegendre{n/2}{3}\chi_{0,2}(n/2)+\chi_4(n/3)+\frac{\sqrt{3}}{2}\mlegendre{n/4}{3}
\]

and get

\[
L(a^{+})=\frac{1+X_3}{2}\sum_n \frac{\chi_4(n)X_n}{n}+\frac{\sqrt{3}(X_2+1)}{4}\sum_n \frac{\chi_3(n)X_n}{n},
\]

thus, $L(1/12,p)\geq 0$ for all $p\equiv 1 \pmod 4$.

For cosines we compute

\[
1-\cos\frac{\pi n}{6}=1-\frac{\sqrt{3}}{2},\frac12,1,\frac32,1+\frac{\sqrt{3}}{2},2,1+\frac{\sqrt{3}}{2},\frac32, 1,\frac12,1-\frac{\sqrt{3}}{2},0,
\]

so that

\[
a_n^{-}=-\frac{\sqrt{3}}{2}\mlegendre{12}{n}+\chi_{0,2}(n)+\frac12\chi_{0,6}(n/2)+\frac32 \chi_{0,3}(n/4)+2\chi_{0,2}(n/6).
\]

Observe that $a_n^{-}$ is always nonnegative, so that we have $\mathbb EL(a^{-})>0$ and $c^{-}(1/12)>0$, which results in $c(1/12)>\frac12.$

When $\alpha=\frac{5}{12}$ we apply the formula $a_n^{\pm}(5/12)=a_{5n}(1/12)$, and for $a^{+}$ we get

\[
L(a^{+})=\frac{1+X_3}{2}\sum_n \frac{\chi_4(n)X_n}{n}-\frac{\sqrt{3}(X_2+1)}{4}\sum_n \frac{\chi_3(n)X_n}{n},
\]

so that $L(a^{+})>0$ for almost all $X_n$ with $X_2=-1$ and $X_3=1$, therefore $c^{+}(5/12)\geq\frac14$. On the other hand, we have $\mlegendre{12}{5}=-1$, so

\[
a_n^{-}=\frac{\sqrt{3}}{2}\mlegendre{12}{n}+\chi_{0,2}(n)+\frac12\chi_{0,6}(n/2)+\frac32 \chi_{0,3}(n/4)+2\chi_{0,2}(n/6),
\]

which gives

\[
L(a^{-})=\frac14(1+(1-X_2)(1-X_3))\sum_n \frac{X_n}{n}+\frac{\sqrt{3}}{2}\sum_n \frac{\mlegendre{12}{n}X_n}{n},
\]

so that $c^{+}(5/12)=1$ and $c(5/12)\geq \frac{1+1/4}{2}=\frac{5}{8}$.

Now, for the last case, where we have to deal with some complex characters.
If $\alpha=\frac15$, then the function $a_n^{-}=1-\cos\frac{2\pi n}{5}$ takes values

\[
\frac{5-\sqrt{5}}{4},\frac{5+\sqrt{5}}{4},\frac{5+\sqrt{5}}{4},\frac{5-\sqrt{5}}{4},0
\]

and we get

\[
L(a^{-})=\sum_n \frac{\left(\frac52\chi_{0,5}(n)-\frac{\sqrt{5}}{2}\mlegendre{n}{5}\right)X_n}{n}=\frac52 H-\frac{\sqrt{5}}{2}T,
\]

where

\[
H=\sum_n \frac{\chi_{0,5}(n)X_n}{n}=\prod_{p\neq 5}\left(1-\frac{X_p}{p}\right)^{-1}
\]

and

\[
T=\sum_n \frac{\mlegendre{n}{5}X_n}{n}=\prod_{p\neq 5}\left(1-\frac{\mlegendre{n}{p}X_p}{p}\right)^{-1}.
\]

The map $X_n\mapsto \lambda(n)X_n$ transforms $H$ into $\frac{4\pi^2}{25H}$ and $T$ into $\frac{4\pi^2}{25T}$, so $c^{-}(1/5)\geq \frac12.$

Expressions for $\sin\frac{2\pi n}{5}$ involve nested radicals, which gives us the formula

\[
\sin\frac{2\pi n}{5}=\frac{A-iB}{2}\kappa(n)+\frac{A+iB}{2}\overline{\kappa(n)},
\]

where $\kappa$ is a Dirichlet character modulo $5$ with $\kappa(2)=i$ and

\[
A=\sqrt{\frac{5+\sqrt{5}}{8}}, B=\sqrt{\frac{5-\sqrt{5}}{8}}.
\]

Therefore, the inequality $L(a^{+})\geq 0$ is equivalent to

\[
\mathrm{Re}\,\left((A-iB)\sum_{n}\frac{\kappa(n)X_n}{n}\right)\geq 0
\]

Now, $A+iB=\sqrt{A^2+B^2}e^{i\varphi}$ with $\varphi=\arctan\left(\frac{B}{A}\right)\approx 0.553$ and our inequality takes the form

\[
\cos\left(-\varphi+\mathrm{arg}\,\left(\sum_n \frac{\kappa(n)X_n}{n}\right)\right)\geq 0.
\]

The sum inside cosine can be expanded into the Euler product:

\[
\sum_n \frac{\kappa(n)X_n}{n}=\prod_p\left(1-\frac{\kappa(p)X_p}{p}\right)^{-1}.
\]

Obviously, if $\kappa(p)$ is real, then $p$ gives no contribution to the sum, and if $\kappa(p)$ is not real (which is equivalent to $p\equiv \pm 2\pmod 5$) then

\[
\left(1-\frac{\kappa(p)X_p}{p}\right)^{-1}=\frac{\exp(\kappa(p)X_p\arctan{\frac{1}{p}})}{\sqrt{1+p^{-2}}}.
\]

From this we deduce the formula

\[
\sum_n \frac{\kappa(n)X_n}{n}=Me^{i\xi},
\]

where $M$ is an almost surely positive real random variable and

\[
\xi=\sum_{p\equiv \pm 2\pmod 5}\frac{\kappa(p)}{i}X_p\arctan{\frac{1}{p}},
\]

which is also real.
Therefore, our desired inequality takes form

\[
\cos(\xi-\varphi)\geq 0.
\]

Notice that if it is not true, then we should have $|\xi-\varphi|>\frac{\pi}{2}$, in which case $|\xi|>\frac{\pi}{2}-\varphi.$

By Chebyshev inequality, probability of this event is at most

\[
\frac{\mathbb E\xi^2}{\left(\frac{\pi}{2}-\varphi\right)^2}\leq \frac{0.3536}{1.077}<\frac13,
\]

because

\[
\mathbb E\xi^2=\sum_{p\equiv \pm 2\pmod 5}\arctan^2{\frac{1}{p}}\approx 0.35355.
\]

From this we obtain $c^{+}(1/5)\geq \frac23$ and so $c(1/5)\geq \frac{7}{12}>\frac12$, as needed.

If $\alpha=\frac25$, then the proof is easier, because in this case we have

\[
a_n^{-}=1-\cos\frac{4\pi n}{5}=\frac52\chi_{0,5}(n)+\frac{\sqrt{5}}{2}\mlegendre{n}{5},
\]

thus

\[
L(a^{-})=\frac52 H+\frac{\sqrt{5}}{2}T\geq 0,
\]

so for all $p\equiv 3 \pmod 4$ we have $L(\frac25,p)\geq 0$. On the other hand, the set of $p\equiv 1 \pmod 4$ with $L(\frac25,p)>0$ has a positive lower density, because we have

\[
\mathbb EL(a^{+})=\sum_n \frac{\sin\frac{4\pi n^2}{5}}{n^2}=\sqrt{\frac{5-\sqrt{5}}{8}}\sum_{n>0}\frac{\mlegendre{n}{5}}{n^2}>0.
\]

From this we get $c(2/5)>\frac12$, which concludes our proof.
\end{proof}

\section{Values of \texorpdfstring{$\alpha$ close to $\frac13$}{alpha close to one-third}}

In the previous section we proved Conjecture 1 for several special values of $\alpha$, using expansion with Dirichlet characters and certain symmetry considerations. Here we are going to use a different approach to the main conjecture. One can notice that all the random variables that we constructed are linear combinations of a fixed sequence of random variables $X_n$ with coefficients that are smooth functions of a parameter $\alpha$. Using this observation, we are going to prove Theorem 3.

Let us define certain class of random variables that appears naturally when we study $L(a^{\pm})$. 
\begin{definition}
Let $\sigma^2>0$. We say that random variable lies in the class $L(\sigma^2)$, if there is an infinite sequence of real numbers $a_1,a_2,\ldots$ with

\[
\sum_{i}a_i^2\leq \sigma^2.
\]

and a sequence $\kappa_1,\kappa_2,\ldots$ of independent Rademacher random variables such that

\[
\eta=\sum_i a_i\kappa_i.
\]
\end{definition}

The next lemma will allow us to control the tail of distribution from the class $L(\sigma^2)$.

\begin{lemma}
Let $\eta$ be in $L(\sigma^2)$. Then for any $T>0$ we have

\[
\mathbb P(\eta\geq T)\leq \exp\left(-\frac{T^2}{2\sigma^2}\right).
\]
\end{lemma}

\begin{proof}
Consider the moment generating function of $\eta$:
 
 \[
 F(t)=\mathbb Ee^{t\eta}.
 \]
 
 It is easy to see that for any positive $t$ the inequality
 
 \[
 \mathbb P(\eta\geq T)\leq e^{-tT}F(t)
 \]
 
 is true. Let us show that $F(t)\leq \exp(t^2\sigma^2/2)$. Indeed, $\eta$ lies in $L(\sigma^2)$, so for some $a_i$ and $\kappa_i$ we have
 
 \[
 F(t)=\mathbb Ee^{t\sum_i a_i\kappa_i}=\prod_i \mathbb Ee^{ta_i\kappa_i}=\prod_i \left(\frac{e^{-a_i t}+e^{a_it}}{2}\right).
 \]
 The expectation of exponent equals the product of expectations due to independence of $\kappa_i$.
 
 Note now that for any real $t$ the inequality
 
 \[
 \frac{e^{-t}+e^t}{2}=1+\frac{t^2}{2}+\frac{t^4}{4!}+\ldots\leq 1+\frac{t^2}{2}+\frac{\left(\frac{t^2}{2}\right)^2}{2!}+\ldots=e^{t^2/2}
 \]
 
 holds. The upper bound is true due to the inequality $(2n)!\geq 2^n n!=(2n)!!$. Therefore, for any real $t$ we have
 
 \[
 F(t)\leq \prod_i e^{a_i^2t^2/2}\leq e^{\sigma^2t^2/2},
 \]
 
 from this we get
 
 \[
 \mathbb P(\eta\geq T)\leq e^{-tT}e^{\sigma^2t^2/2}.
 \]
 
 Choosing $t=T/\sigma^2$, we prove the desired inequality.
 
\end{proof}

 From this lemma we immediately deduce the bound for probability of negativity of random variables that are close to $\exp(\eta)$ for some $\eta\in L(\sigma^2)$.
 
 \begin{corollary}
 Assume that real random variables $X$ and $Y$ satisfy $X=e^{\eta}$ for some $\eta \in L(\sigma^2)$ and 

\[
\mathbb E(X-Y)^2\leq D.
\]
Then for any $0<u<1$ we have

\[
\mathbb P(Y\leq 0)\leq \exp\left(-\frac{\ln^2 u}{8\sigma^2}\right)+\frac{D}{u}.
\]

\end{corollary}

\begin{proof}
Indeed, if $Y\leq 0$ then either $(X-Y)^2\geq u$ or $X^2 \leq u$. Probability of the first event is at most $\frac{D}{u}$ due to Markov's inequality. On the other hand, the second event implies that $\eta\leq-\frac{\ln u}{2}$, therefore $-\eta\geq\frac{\ln u}{2}$. Notice now that if $\eta\in L(\sigma^2)$ then the same is true for $-\eta$, as one can choose $-a_i$ instead of $a_i$ in the defining formula. Thus, Lemma 6 implies that the probability of second event is at most $\exp\left(-\frac{\ln^2 u}{8\sigma^2}\right)$. This proves the desired estimate.
\end{proof}
Let us show now that random variables $L(a^{\pm}(1/3))$ are proportional to exponents of certain random variables from the class $L(\sigma^2).$

\begin{lemma}
For $\sigma^2=0.395$ there are $\eta_1$ and $\eta_2$ from $L(\sigma^2)$ such that

\[
L(a^{+}(1/3))=\frac{\pi}{\sqrt{3}}e^{\eta_1}
\]

and

\[
L(a^{-}(1/3))=\frac{\pi}{3}e^{\eta_2}.
\]
\end{lemma}

\begin{proof}
Due to the formulas (\ref{aplus3}) and (\ref{amin3}) we have
 
 \[
 L(a^{+}(1/3))=\frac{3}{2}\sum_{n>0}\frac{X_n\chi_{0,3}(n)}{n}=\frac32\prod_{p\neq 3}\left(1-\frac{X_p}{p}\right)^{-1}
 \]
 
 and
 
 \[
 L(a^{-}(1/3))=\frac{\sqrt{3}}{2}\sum_{n>0}\frac{\mlegendre{n}{3}X_n}{n}=\frac{\sqrt{3}}{2}\prod_{p\neq 3}\left(1-\frac{X_p\mlegendre{p}{3})}{p}\right)^{-1}
 \]
 
Let us observe that for any $\varepsilon=\pm 1$ the identities

\[
\left(1-\frac{\varepsilon}{p}\right)\left(1+\frac{\varepsilon}{p}\right)=1-\frac{1}{p^2}
\]

and

\[
\left(1-\frac{\varepsilon}{p}\right)\left(1+\frac{\varepsilon}{p}\right)^{-1}=\left(\frac{p-1}{p+1}\right)^{\varepsilon}
\]
hold. Multiplying this two equalities and taking a square root, we deduce

\[
1-\frac{\varepsilon}{p}=\left(\frac{p-1}{p+1}\right)^{\varepsilon/2}\left(1-\frac{1}{p^2}\right)^{1/2}.
\]

If we now choose $\varepsilon$ to be $X_p$ and $X_p\mlegendre{p}{3}$, we obtain the relations

\[
L(a^{-}(1/3))=A_1e^{\eta_1}
\]
and
\[
L(a^{+}(1/3))=A_2e^{\eta_2},
\]

where

\[
A_1=\frac32 \prod_{p\neq 3}\left(1-\frac{1}{p^2}\right)^{1/2}=\frac32 \sqrt{\frac{\pi^2}{6}(1-1/9)}=\frac32 \sqrt{\frac{4\pi^2}{27}}=\frac{\pi}{\sqrt{3}},
\]

\[
A_2=\frac{A_1}{\sqrt{3}}=\frac{\pi}{3},
\]

\[
\eta_1=\sum_{p\neq 3}\frac{1}{2}\ln\left(\frac{p-1}{p+1}\right)X_p
\]

and

\[
\eta_2=\sum_{p\neq 3}\frac{1}{2}\ln\left(\frac{p-1}{p+1}\right)\mlegendre{p}{3}X_p.
\]

Direct computation shows that the inequality

\[
\sum_{p\neq 3}\frac{1}{4}\ln^2\left(\frac{p-1}{p+1}\right)<0.395,
\]

is true, which concludes the proof of lemma.
\end{proof}
 Now we are going to show that if $\alpha$ is close to $\frac13$ then random variables $L(a^{\pm}(\alpha))$ and $L(a^{\pm}(1/3))$ are in some sense close to each other. But first let us prove two auxiliary propostitions:

\begin{lemma}
Let $\tau(n)=\sum\limits_{d\mid n}1$ be the divisor function. Then for any complex $s$ with $\mathrm{Re}\,s>1$ the equality

\[
\sum_{n=1}^{+\infty}\frac{\tau(n^2)}{n^s}=\frac{\zeta(s)^3}{\zeta(2s)}
\]

holds.
\end{lemma}
\begin{proof}
Due to multiplicativity of $\tau(n^2),$ we have

\[
\sum_{n=1}^{+\infty}\frac{\tau(n^2)}{n^s}=\prod_p \left(1+3p^{-s}+5p^{-2s}+\ldots\right)=\prod_p\left(\frac{1-p^{-2s}}{(1-p^{-s})^3}\right),
\]

which gives the desired equality.
\end{proof}
Using this lemma, we will prove the following bound for $L^2-$distance between two random variables of certain type.
\begin{lemma}
Let $f:\mathbb R\to \mathbb R$ --- be a function with Lipschitz constant $L$ such that $|f(x)|\leq C$ for any real $x$. Then for any pair of real $\alpha$ and $\beta$ we have

\[
\mathbb E(L(f(\alpha))-L(f(\beta)))^2\leq 92|\alpha-\beta|^{2/3}L^{2/3}C^{4/3},
\]

where for any real $\gamma$ the sequence $f_n(\gamma)$ is defined by formula $f_n(\gamma)=f(n\gamma)$ and $L(f(\gamma))$ then given by Definition 3 for $a_n=f_n(\gamma).$
\end{lemma}
\begin{proof}
By the definition of $L(f(\alpha))-L(f(\beta))$, we have

\[
\mathbb E(L(f(\alpha))-L(f(\beta)))^2=\sum_{n,m>0}\frac{(f(n\alpha)-f(n\beta))(f(m\alpha)-f(m\beta))}{nm}\mathbb E(X_nX_m).
\]

Let us notice that $\mathbb E(X_nX_m)=1$ if $nm=d^2$ for some integer $d$ and $0$ otherwise. Splitting the resulting sum according to range of $d$ into $d>A$ and $d\leq A$ parts, we get

\[
\mathbb E(L(f(\alpha))-L(f(\beta)))^2=\sum_{nm=d^2, d\leq A}\frac{(f(n\alpha)-f(n\beta))(f(m\alpha)-f(m\beta))}{nm}+
\]
\[+\sum_{nm=d^2, d>A}\frac{(f(n\alpha)-f(n\beta))(f(m\alpha)-f(m\beta))}{nm}.
\]

Every summand of the first sum is at most $L^2|\alpha-\beta|^2$ due to the Lipschitz property, while in the second sum every summand is bounded by $\frac{4C^2}{d^2}$. Furthermore, any $d$ corresponds to exactly $\tau(d^2)$ summands. From this we obtain the inequality

\[
\mathbb E(L(f(\alpha))-L(f(\beta)))^2\leq |\alpha-\beta|^2L^2\sum_{d\leq A}\tau(d^2)+4C^2\sum_{d>A}\frac{\tau(d^2)}{d^2}.
\]

To get a more convenient estimate, we multiply every summand of the first sum by $(A/d)^{4/3}$ and each summand in the second sum by $(d/A)^{2/3}$. Obviously, both sums will not decrease. Thus,

\[
\sum_{d\leq A}\tau(d^2)\leq A^{4/3}\sum_{d\leq A}\frac{\tau(n^2)}{n^{4/3}}=A^{4/3}s_1(A),
\]

due to Lemma 8. Similarly,

\[
\sum_{d>A}\frac{\tau(d^2)}{d^2}\leq A^{-2/3}\sum_{d>A} \frac{\tau(d^2)}{d^{4/3}}=A^{-2/3}s_2(A).
\]
Observing that $s_1(A)+s_2(A)=\frac{\zeta(4/3)^3}{\zeta(8/3)}$ by lemma 8, we get for all positive $A$
\[
\mathbb E(L(f(\alpha))-L(f(\beta)))^2\leq |\alpha-\beta|^2L^2A^{4/3}s_1(A)+4C^2A^{-2/3}s_2(A)\leq\]
\[\leq\frac{\zeta(4/3)^3}{\zeta(8/3)}\max(|\alpha-\beta|^2L^2A^{4/3},4C^2A^{-2/3})
\]
Choosing $A=2CL^{-1}|\alpha-\beta|^{-1}$ and using relation $\frac{\zeta(4/3)^3}{\zeta(8/3)}2^{4/3}<92,$ we obtain the desired estimate.
\end{proof}

Applying Lemma 9 to the functions $f(x)=\sin(2\pi x)$ and $f(x)=\cos(2\pi x)$ with constants $C=1$ and $L=2\pi$, we obtain the inequality

\begin{equation}
\label{fund}
\mathbb E(L(a^{\pm}(\alpha))-L(a^{\pm}(\beta)))^2\leq 313.3|\alpha-\beta|^{2/3}.
\end{equation}
Now we have enough instruments to prove Theorem 2.

\begin{proof}[of Theorem 2]
Due to Lemma 5 and inequality \ref{fund}, the estimates

\[
\mathbb E(e^{\eta_1}-\frac{\sqrt{3}}{\pi}L(a^{-}(\alpha)))^2=\frac{3}{\pi^2}\mathbb E(L(a^{-}(1/3))-L(a^{-}(\alpha)))^2\leq \frac{926.9}{\pi^2}|\alpha-1/3|^{2/3}\leq 94|\alpha-1/3|^{2/3}
\]

and
\[
\mathbb E(e^{\eta_2}-\frac{3}{\pi}L(a^{+}(\alpha)))^2=\frac{9}{\pi^2}\mathbb E(L(a^{+}(1/3))-L(a^{+}(\alpha)))^2\leq \frac{2780.7}{\pi^2}|\alpha-1/3|^{2/3}\leq 282|\alpha-1/3|^{2/3}
\]
hold.

From this for $|\alpha-1/3|<2\cdot10^{-6}$ we deduce, using Corollary 4, that the probability of negativity of $L(a^{-}(\alpha))$ satisfies for any $0<u<1$ the inequality

\[
\mathbb P(L(a^{-}(\alpha))\leq0)\leq \exp\left(-\frac{\ln^2 u}{3.16}\right)+\frac{0.015}{u}
\]

and similarly

\[
\mathbb P(L(a^{+}(\alpha)\leq0))\leq \exp\left(-\frac{\ln^2 u}{3.16}\right)+\frac{0.0447}{u}.
\]

Computation shows that the minimal values of resulting expressions are attained at the points $u_{-}\approx 0.0756$ and $u_{+}\approx 0.12957$. Values of these expressions in corresponding points are less than or equal to $0.32$ and $0.612$ respectively. From this we obtain the inequality

\[
c(\alpha)\geq 1-\frac{0.32+0.612}{2}=0.534
\]
in the desired range. Therefore, if $\alpha$ lies inside $2\cdot10^{-6}$-neighbourhood of $1/3$, then the sum $L(\alpha,p)$ is nonnegative for at least 53.4\% of all primes, which concludes the proof.
\end{proof}

\section{Conclusion and open problems}

In this paper we were able to reduce Conjecture 1 to the study of one fixed random variable and partially prove the conjecture. Although the progress we made simplifies the original problem, there are still a lot of questions one may ask even beyond the Conjecture 1. For example, is it true, that the $\liminf$ in \ref{goal} can be replaced by $\lim$? If so, what are the properties of this limit as a function of $\alpha$? Also, from section 4 we see that sometimes probability of events $L(a^{-})=0$ or $L(a^{+})=0$ is nonzero: for example, if $X_2=X_3=-1$ then $L(a^{+}(1/6))=L(a^{-}(1/6))=0$. Is it possible to describe all $\alpha$ with $\mathbb P(L(a^{+}(\alpha))=0)+\mathbb P(L(a^{-}(\alpha))=0)>0$ and is this set finite? On the other hand, one might also ask about a computational perspectives of our results. As we use Siegel-Walfisz theorem in the proof of Theorem 5, the proof provides no method to estimate the rate of convergence of resulting distribution. It is probably possible to overcome this obstacle  with the help of Page theorems but we don't know if this can be useful for numerical verification of inequality $c(\alpha)>\frac12.$

\end{document}